\documentclass[12pt]{amsart}
\usepackage{amsmath,amssymb,amsfonts}
\usepackage{xcolor}
\usepackage[colorlinks=true,linkcolor=red,citecolor=blue]{hyperref}
\newtheorem{theorem}{Theorem}[section]

\newcommand\myenum[1]

\begin{document}
\title[A note on the ideals $E^a$ which are discrete]{A note on the ideals $E^a$ which are discrete}

\author{H. Khabaoui}
\address{Hassan Khabaoui, Universit\'{e} Moulay Ismail, Facult\'{e} des Sciences, D\'{e}partement de Math\'{e}matiques, B.P. 11201 Zitoune, Mekn\`{e}s, Morocco}
\email{khabaoui.hassan2@gmail.com}

\author{K. El fahri}
\address{Kamal El fahri, Department of Mathematics, Ibn Zohr University, Faculty of Sciences, Agadir, "Moroccan Association of Ordered Structures, Operators Theory, Applications and Sustainable Development (MAOSOTA)", Morocco.}
\email{kamalelfahri@gmail.com }

\author{Jawad H'michane}
\address[J. H'michane]{Engineering Sciences Lab, ENSA, B.P 241, Ibn Tofail University, K\'{e}nitra,  "Moroccan Association of Ordered Structures, Operators Theory, Applications and Sustainable Development (MAOSOTA)", Morocco}

\begin{abstract}
This note looks at two mistakes in the paper "Some characterizations of the ideals $E^{a}$ which are discrete" by Khabaoui et al., published in "Rend. Circ. Mat. Palermo, II. Ser 72, 1325–1336 (2023)."
\end{abstract}

\maketitle

\section{Main results}

This note looks at two mistakes in the paper "Some characterizations of the ideals $E^{a}$ which are discrete" by Khabaoui et al., published in "Rend. Circ. Mat. Palermo, II. Ser 72, 1325–1336 (2023)." We will demonstrate and correct two mistakes in the  article "Some characterizations of the ideals $E^{a}$ which are discrete" by Khabaoui et al., published in "Rend. Circ. Mat. Palermo, II. Ser 72, 1325–1336 (2023)". Fortunately it is only the proofs that are incorrect and not the results.  The same terms and notations are used in this note as in the original paper, and so it has been decided not to reintroduce them here.

 The first error is in Theorem 3.7 \cite{Khabaoui}, where in the proof of the implication $(1) \Rightarrow (2)$, in view of $||y_n||=1$ for every $n$, it follows that the operator  $R :  \ell^{\infty}   \to  F$ is not well defined. So the given proof of this implication is not true. The correct proof of Theorem 3.7 \cite{Khabaoui} can be given as follows:
\begin{theorem}\label{duarecip}	
Let $E$ and $F$ be two Banach lattice such that $F'$ is order continuous. The following statements are equivalent:
	\begin{enumerate}
		\item Every regular operator $T : E \longrightarrow F$ is Lwcc, whenever its adjoint operator $T' : F' \longrightarrow E'$ is Lwcc.
		\item One of  the following conditions is holds:
			\begin{enumerate}
			\item  $F$ is order continuous.
			\item $ E^{a}$ is discrete.
		\end{enumerate}
	\end{enumerate}
\end{theorem}
	\begin{proof}
$1) \Rightarrow 2)$ Assume by way of contradiction that $F$ is not order continuous and that $E^{a}$ is not discrete. Since $F$ is not order continuous, then it follows from Theorem 2.4.2 \cite{Mey} that there exist some $y\in  F^{+}$ and a disjoint sequence $(y_n) \in \left[0,y \right]$ such that $||y_n||=1$, for every $n$. As $ E^{a}$ is not discrete, by Corollary 3.1 \cite{Khabaoui}  and the proof of Theorem \cite{Wick} we infer that there exist a weakly null sequence $(x_n)$ of $E$ which is L-weakly compact, $0\leq h\in E'$, $g,g_n \in \left[-h,h \right]$ such that  $g_n \overset{w^*}\longrightarrow g$ and $g_n(x_n)>\varepsilon $ for every $n$ and  for some $\varepsilon >0 $. We consider the operators $S$ and $R$ defined respectively by:
$$
\begin{array}{ccccc}
S & : & E  & \to & c_0 \\
& & x & \mapsto & \left( (g-g_n)(x)\right)_n  \\
\end{array}
$$
and
$$ \begin{array}{ccccc}
R & : & c_0  & \to & F \\
& & (\alpha_{n}) & \mapsto & \sum_{n=1}^{\infty} \alpha_{n} y_n.  \\
\end{array}$$
We put $T=R\circ S$, the operator $T$ is not Lwcc. Indeed, for a fixed integer $n$ and  since $(y_k)$ is a disjoint sequence, then for every integer $m$, we have
$$|\sum_{k=1}^{m} \left(g-g_k \right)(x_n)y_k |= \sum_{k=1}^{m} |\left(g-g_k \right)(x_n)|y_k$$
and by taking norm limits we see that
\begin{eqnarray*}
	|T(x_n)| &=& |\sum_{k=1}^{\infty} \left(g-g_k \right)(x_n)y_k |\\
	&=& \sum_{k=1}^{\infty} |\left(g-g_k \right)(x_n)|y_k \\
	&\geq& |\left(g-g_n \right)(x_n)|y_n
\end{eqnarray*}
hence, $||T(x_n)||\geq |\left(g-g_n \right)(x_n)| $ and since $\left(g-g_n \right)(x_n)\nrightarrow 0$, then $||T(x_n)||\nrightarrow 0$. That is $T$ is not Lwcc. Since $\ell^1$ has the Schur property, then $T'$ is Dunford Pettis and hence $T'$ is Lwcc.

$2-a) \Rightarrow 1)$ Let $T : E \longrightarrow F$ be a regular operator such that $T'$ is Lwcc. Since $F'$ is order continuous, then it follows from Theorem 3.5 \cite{Khabaoui} that $T'$ is AM-compact. As $F$ is order continuous, then by Proposition 3.7.23 \cite{Mey} the operator $T$ is AM-compact and hence $T$ is Lwcc.

$2-b) \Rightarrow 1)$ Obvious.
	\end{proof}		
	\color{black}

The second error is in the proof of Theorem 3.8 \cite{Khabaoui}, where in this result, it was proposed to characterize Banach lattices such that operators dominated by Lwcc operator are Lwcc. But there was an error in the proof of the above-mentioned theorem. In fact, the operator $S$  which  is used in the implication $(1) \Rightarrow(2)$  unfortunately is not well defined, because $((f_n-f)(x))\nrightarrow 0$  for each $x \in E$. Now we are in a position to give the correct proof of this false  implication.

\begin{theorem}\label{Ae}
Let $E$ and $F$ be Banach lattices. Then, the following statements are equivalent:
 \begin{enumerate}
 \item For all operators $S, T : E \longrightarrow F$ such that $0 \leq S \leq T$ and $T$ is Lwcc, the operator $S$ is Lwcc.
 \item One of the following conditions is holds:
  \begin{enumerate}
   \item $E^a$ is discrete;
   \item $F$ is order continuous.
  \end{enumerate}
 \end{enumerate}
\end{theorem}
	\begin{proof}
$1)\Rightarrow 2)$ Suppose that neither $E^a$ is discrete nor $F$ is order continuous. Since $E^a$ is not discrete, by Corollary  3.1 \cite{Khabaoui}  and the proof of Theorem \cite{Wick} we infer that there exist a weakly null sequence $(x_n)$ of $E$ which is L-weakly compact, $0\leq h\in E' $, $g,g_n \in \left[-h,h \right]$ such that  $g_n \overset{w^*}\longrightarrow g$ and $g_n(x_n)>\varepsilon $ for every $n$ and  for some $\varepsilon >0 $. On the other hand, since $F$ is not order continuous then by Theorem 2.4.2 \cite{Mey} there exist a positive sequence $(y_n)$ of $F$  and $ y \in F$ such that $ 0 \leq y_n \leq y$ for every $n$ and $||y_n||=1$ for every integer $n$.\\
Now, we consider the following two operators $S,T : E \rightarrow F$ defined as follows $$S(x) =\sum_{i=1}^{\infty} (g(x)-g_i(x))y_i + 2h(x)y$$ and $$T(x) = 4h(x)y;$$
clearly that $0\leq S \leq T$ and that $T$ is Lwcc but $S$ is not Lwcc. Indeed, we have for every $n$ $$ ||S(x_n)-2h(x_n)y|| \geq |g(x_n)-g_n(x_n)|$$ and since $(x_n)$ is weakly null then $h(x_n)y\rightarrow 0$ and $g(x_n)\rightarrow 0$. And since  $|g_n(x_n)|>\varepsilon $ for every $n$, then  $|g(x_n)-g_n(x_n)|\nrightarrow 0$ which implies that $\|S(x_n)\|\nrightarrow0$ and hence $S$ is not Lwcc.

$2-a)\Rightarrow 1)$ Let $(x_n)$ be a weakly null sequence of $E$ which is L-weakly compact, by Corollary 3.1 \cite{Khabaoui} the sequence $(|x_n|)$ is weakly null and L-weakly compact and since $T$ is Lwcc then $\|T(|x_n|)\| {\overset{ }{\longrightarrow}} 0$. As $|S(x_n)|\leq T(|x_n|)$ for every $n$, then $\|S(x_n)\| {\overset{ }{\longrightarrow}} 0$, that is $S$ is Lwcc.

$2-b) \Rightarrow 1)$ Since $T$  is Lwcc, then by Theorem 3.5 \cite{Khabaoui} $T_{|E^a}$ is an AM-compact operator and since $F$  is order continuous it follows from Theorem 2.10 \cite{Zra} that $S_{|E^a}$ is an AM-compact operator and hence by Theorem 3.5 \cite{Khabaoui} $S$ is Lwcc.
\end{proof}



\begin{thebibliography}{9}
\bibitem{Zra} B. Aqzzouz, R. Nouira and L. Zraoula, Compactness properties of operators dominated by AM-compact operators. Proceedings of the American mathematical society. Volume 135, Number 4, April 2007, Pages 1151–1157, S 0002-9939(06)08585-6 (2006).
    
\bibitem{Dod} P.G. Dodds and D.H. Fremlin, Compact operators on Banach lattices, Israel J. Math. 34 (1979), 287-320.
    
\bibitem{Khabaoui} H. Khabaoui, K. El Fahri, S. Laachachi and J. H'michane, Some characterizations of the ideals $E^{a}$ which are discrete. Rend. Circ. Mat. Palermo, II. Ser 72, 1325-1336 (2023).

\bibitem{Mey} P. Meyer-Nieberg, Banach Lattices, Universitext, Springer-Verlag, Berlin, 1991.
\bibitem{Wick}  A. W. Wickstead, Converses for the Dodds-Fremlin and Kalton-Saab theorems, Math. Proc. Camb. Phil. Soc. 120, 175–179 (1996).
\end{thebibliography}
\end{document}